\documentclass[12pt,leqno]{amsart}
\usepackage{amsmath,amsfonts,amssymb,amsthm}
\usepackage{graphicx}
\graphicspath{ }
\usepackage{wrapfig}
\usepackage{accents}
\usepackage{caption}
\usepackage{subcaption}
\usepackage{calligra}

\numberwithin{figure}{section}
\newcommand\norm[1]{\left\lVert#1\right\rVert}

\theoremstyle{plain}
\newtheorem{theorem}{Theorem}
\newtheorem{lemma}[theorem]{Lemma}
\newtheorem{proposition}[theorem]{Proposition}
\newtheorem{corollary}{Corollary}[theorem]
\theoremstyle{definition}
\newtheorem{definition}{Definition}[section]

\theoremstyle{remark}

\usepackage{mathtools}
\begin{document}
\title[Integral Liouville theorem in manifold]{Integral Liouville theorem in a complete Riemannian manifold}
\author[A. A. Shaikh and C. K. Mondal]{Absos Ali Shaikh$^{1*}$ and Chandan Kumar Mondal$^2$}
%\address{\noindent\newline $^{1,2}$Department of Mathematics,\newline University of
%Burdwan, Golapbag,\newline Burdwan-713104,\newline West Bengal, India}
%\email{aask2003@yahoo.co.in, aashaikh@math.buruniv.ac.in}
%\email{chan.alge@gmail.com}

\begin{abstract}
If the Killing vector field in a Riemannian manifold is the gradient of a smooth
real valued function, then it is called Killing potential. In this paper we have deduced a
necessary condition for the existence of Killing potential in a complete Riemannian manifold.
Yau proved the Liouville theorem of harmonic function in a Riemannian manifold using gradient
estimation and after that many authors have generalized this concept and investigated various
types of Liouville theorems of harmonic functions. In this article we have also proved a Liouville
theorem in integral form of harmonic functions in a complete Riemannian manifold. Finally
we have studied the behaviour of harmonic functions in a complete Riemannian manifold that
satisfies some gradient Ricci solition and showed that harmonic function is a constant multiple
of distance function along some geodesics.
\end{abstract}
\noindent\footnotetext{ $^*$ Corresponding author.\\
$\mathbf{2010}$\hspace{5pt}Mathematics\; Subject\; Classification: 53C21, 53C43, 58E20,  58J05.\\ 
{Key words and phrases: Killing vector field; Killing potential; harmonic function; subharmonic function; convex function; Neumann-Poincar\'{e} inequality; Potential function.} }
\maketitle
\section{\textbf{Introduction}}
The field of harmonic analysis in a Riemannian manifold was developed by the seminal work of Yau \cite{YAU75} in 1985. The existence of harmonic function is very significant in differential geometry as it depends on the topological and geometrical properties of a manifold. In that paper Yau proved Liouville theorem for positive harmonic functions on a complete manifold with positive semi-definite Ricci curvature. In particular, he showed that a positive harmonic function on such a manifold must be constant. Since then the study of various aspects of harmonic functions become an active research area in the field of geometric analysis and various authors proved many important results in this field, for reference see (\cite{MT14}, \cite{XU16}, \cite{CH83}, \cite{CM97}). Later Yau's Liouville theorem has been generalized and extended. Cheng-Yau \cite{CY75} generalized the Liouville theorem and proved Harnack inequality. Andersen \cite{AN83} and Sullivan \cite{SU83} proved that simply connected Riemannian manifolds of negative sectional curvature admits many bounded harmonic functions. For more results see \cite{CM96}. Recently, this result was extended by Xia \cite{XI14} in Finsler manifolds under the condition that the weighted Ricci curvature has a lower bound and by Zhang-Zhu \cite{ZZ12} to the Alexandrov spaces. Liouville theorem has also been extended in case of subharmonic functions, see \cite{CKS98}. \\
\indent Motivating by these works, we have proved a Liouville theorem in integral form and studied
the behaviour of harmonic functions in a complete Riemannian manifold. In particular, we have
showed that in a complete Riemannian manifold $(M,g)$ if a harmonic function takes negative
value at a point, where the Riemannian metric $g$ is rotationally symmetric and whose injective
radius is greater than one, then the integration of the harmonic function in any open ball with
center at that particular point and radius less than the injective radius of that point remains
invariant. We also showed that in a complete non-compact Riemannian manifold which satisfies
gradient Ricci solition, harmonic function is equivalent to the distance function, i.e., they differ by some
constant terms. We have also studied the Killing vector field in a Riemannian manifold and
proved the non-existence of Killing potential in a complete Riemannian manifold with positive
lower bound Ricci curvature.\\
\indent The paper is organized as follows: Section 2 deals with some preliminaries concept of Riemannian
manifold and some definitions, which are needed for the rest of this paper. Section 3
deals with the study of the Killing vector field and Killing potential in a complete Riemannian
manifold. In this section we have deduced a necessary condition for the existence of Killing potential
in a complete Riemannian manifold. Section 3 is devoted to the study of some properties
of harmonic functions and proved a Liouville theorem, which is the main theorem in this article,
for a non-positive harmonic function. In this section we have also deduced some inequalities of
harmonic functions and superharmonic functions in a complete Riemannian manifold. In the
last section, we have investigated the behavior of harmonic functions in a complete Riemannian
manifold satisfying gradient Ricci solition and showed that such harmonic function reduces to
a simple form along some geodesics.
%%%%%%%%%%%%%%%%%%%%%%%%%%%%%%%%%%%%%%%%%%%%%%%%%%%%%%%%%%%%%%%%%%%%%%%%%%%%%%%%%%%%%%%%%%%%%%%%%%%%%%%%%%%%%%%%%%%%%%%%%%%%%%%%%%%%%%%%%%%%%%%%%%%%%%%%%%%%%%%%%%%%%%%%%%%%%%%%%%%%%%%%%%%%%%%%%%%%%%%%%%%%%%%%%%%%%%%%%%%%%%%%%%%%%%%%%%%%%%%%%
\section{\textbf{Preliminaries and definitions}}
%%%%%%%%%%%%%%%%%%%%%%%%%%%%%%%%%%%%%%%%%%%%%%%%%%%%%%%%%%%%%%%%%%%%%%%%%%%%%%%%%%%%%%%%%%%%%%%%%%%%%%%%%%%%%%%%%%%%%%%%%%%%%%%%%%%%%%%%%%%%%%%%%%%%%%%%%%%%%%%%%%%%%%%%%%%%%%%%%%%%%%%%%%%%%%%%%%%%%%%%%%%%%%%%%%%%%%%%%%%%%%%%%%%%%%%%%%%%%%%%%
In this section we have discussed some basic facts of a Riemannian manifold $(M,g)$, which will be used throughout this paper (for reference see \cite{PE06}). The tangent space at the point $p\in M$ is denoted by $T_pM$ and the tangent bundle is defined by $TM=\cup_{p\in M}T_pM$. The length $l(\gamma)$ of the curve $\gamma:[a,b]\rightarrow M$ is given by
\begin{eqnarray*}
l(\gamma)&=&\int_{a}^{b}\sqrt{g_{\gamma(t)}({\gamma}'(t),{\gamma}'(t))}\ dt\\
&=&\int_{a}^{b}\norm {{\gamma}'(t)}dt.
\end{eqnarray*}
A curve $\gamma$ is said to be a geodesic if $\nabla_{{\gamma}'(t)}{\gamma}'(t)=0\ \forall t\in [a,b]$, where $\nabla$ is the Riemannian connection of $g$.
For any point $p\in M$, the exponential map $exp_p:V_p\rightarrow M$ is defined by
$$exp_p(u)=\gamma_u(1),$$
where $\gamma_u$ is a geodesic with $\gamma(0)=p$ and ${\gamma}'_u(0)=u$ and $V_p$ is a collection of vectors of $T_pM$ such that for each element $u\in V_p$, the geodesic with initial tangent vector $u$ is defined on $[0,1]$. It can be easily seen that for a geodesic $\gamma$, the norm of a tangent vector is constant, i.e., $\norm {{\gamma}'(t)}$ is constant. If the tangent vector of a geodesic is of unit norm, then the geodesic is called normal. If the exponential map exp is defined at all points of $T_pM$ for each $p\in M$, then $M$ is called complete. Hopf-Rinow theorem provides some equivalent cases for the completeness of $M$. Let $x$, $y\in M$. The distance between $p$ and $q$ is defined by
$$d(x,y)=\inf\{l(\gamma):\gamma \text{ be a curve joining }x \text{ and }y\}.$$
A geodesic $\gamma$ joining $x$ and $y$ is called minimal if $l(\gamma)=d(x,y)$. Hopf-Rinow theorem guarantees the existence of minimal geodesics between two points of $M$. The injective radius $inj(x)$ of a point $x\in M$ is the largest $r>0$ for which any geodesic of length less than $r$ and having $x$ as an endpoint is length minimizing. Let $x\in M$ and $0<r<inj(x)$, then in $B_r(x)$ we can define the polar coordinates of $\mathbb{R}^n$ as $(r,\theta^1,\cdots,\theta^{n-1})$, the metric of $M$ can be expressed in polar form, i.e.,
\begin{equation}\label{eq12}
ds^2=dr^2+\sum_{i,j}g_{ij}d\theta^id\theta^j=dr^2+h(r,\theta)^2d\Theta^2
\end{equation}
on $B_r(x)$, where $g_{ij}=g(\frac{\partial}{\partial \theta^i},\frac{\partial}{\partial \theta^j})$ and $d\Theta^2$ is the canonical metric on the unit sphere of $T_xM$. Hence $r(y)$ denotes the geodesic distance from $x$ to $y$ for any $y\in B_r(x)$. The Riemannian volume element of $\partial B_r(x)$ can be expressed as $dS_r=\sqrt{D(r,\Theta)}d\theta^1\cdots d\theta^{n-1},$ where $D=det(g_{ij})$. If the function $h$ in (\ref{eq12}) depends only on $r$ but not on $\theta$, then $M$ is called rotationally symmetric \cite{CHO83} at $x$. \\
\indent  A smooth vector field is a smooth function $X:M\rightarrow TM$ such that $\pi\circ X=id_M$, where $\pi:TM\rightarrow M$ is the projection map. Given any vector field $X\in\chi(M)$ and a covariant tensor field $\omega$ of order $r$ on a Riemannian manifold, the Lie derivative of $\omega$ with respect to $X$ is defined by
$$(\mathcal{L}_X\omega)(X_1,\dots,X_r)=X(\omega(X_1,\dots,X_r))-\sum_{i=1}^{r}\omega(X_1,\dots,[X,X_i],\dots,X_n),$$
where $ X_i\in \chi(M)\text{ for }i=1,\dots,r.$ In particular, when $\omega=g$, then 
$$(\mathcal{L}_Xg)(Y,Z)=g(\nabla_YX,Z)+g(Y,\nabla_ZX) \text{ for }Y,Z\in \chi(M).$$\\
\indent Gromoll and Meyer \cite{GM69} introduced the notion of pole in a Riemannian manifold. A point $o\in M$ is called a pole of $M$ if the exponential map at $o$ is a global diffeomorphism and a manifold $M$ with a pole $o$ is denoted by $(M,o)$. If a manifold possesses a pole then it can be easily seen that the manifold is complete although the converse does not hold always, see \cite{14}. Simply connected complete Riemannian manifold with non-positive sectional curvature and a two-dimensional sphere with one point removed are the examples of Riemannian manifolds with a pole.\\
\indent  The gradient of a smooth function $u:M\rightarrow\mathbb{R}$ at the point $p\in M$ is defined by $\nabla u(p)=g^{ij}\frac{\partial u}{\partial x^j}\frac{\partial}{\partial x^i}\mid_p.$ It is the unique vector field such that $g( \nabla u,X)=X(u)$ for any smooth vector field $X$ in $M$. The Hessian $Hess(u)$ is the symmetric $(0,2)$-tensor field, defined by $\nabla^2u(X,Y)=Hess(u)(X,Y)=g(\nabla_X\nabla u,Y)$ for all smooth vector fields $X,Y$ of $M$. In local coordinates this can be written as
$$(\nabla^2u)_{ij}=\partial_{ij}u-\Gamma^k_{ij}\partial_ku.$$
Hence in general one can write
$$|\nabla^ku|^2=g^{i_1j_1}\cdots g^{i_kj_k}(\nabla^ku)_{i_1\cdots i_k}(\nabla^ku)_{j_1\cdots j_k}.$$
For an integer $k$ and a real number $p\geq 1$, we use the following notation
$$C^\infty_k(M)=\Big\{u\in C^\infty(M): \int_{M}|\nabla^ju|^pdV<\infty, \forall j=0,\cdots,k\Big\}.$$
\begin{definition}
The Sobolev space $W^{k,p}(M)$ is the completion of $C^p_k(M)$ with respect to the norm
$$\norm{u}_{W^{k,p}}=\sum_{j=0}^{k}\Big(\int_M|\nabla^ju|^pdV\Big)^{1/p}.$$
\end{definition}
A function $u\in C^\infty(M)$ is said to belong in the class $W^{k,p}_{loc}(M)$ if for each point $x\in M$ there exists a function $\varphi\in C^\infty(M)$ with compact support such that $\varphi\circ u$ belongs to $W^{k,p}(M).$
 For a vector field $X$, the divergence of $X$ is defined by
$$div(X)=\frac{1}{\sqrt{g}}\frac{\partial }{\partial x^j}\sqrt{g}X^j,$$
where $g=det(g_{ij})$ and $X=X^j\frac{\partial}{\partial x^j}$. The Laplacian of $u$ is defined by $\Delta u=div(\nabla u)$. 
\begin{definition}\cite{YAU75}
A $C^2$-function $u:M\rightarrow\mathbb{R}$ is said to be harmonic if $\Delta u=0$. The function $u$ is called subharmonic (resp. superharmonic) if $\Delta \geq 0$ (resp. $\Delta u\leq 0$).
\end{definition}
\begin{definition}\cite{UDR94}
 A real valued function $u$ on $M$ is called convex if for every geodesic $\gamma:[a,b]\rightarrow M$, the following inequality holds
\begin{equation*}
u\circ\gamma((1-t)a+tb)\leq (1-t)u\circ\gamma(a)+tu\circ\gamma(b)\quad \forall t\in [0,1],
\end{equation*}
or if $u$ is differentiable, then 
\begin{equation*}
g(\nabla u,X)_x\leq u(exp_x\nabla u)-u(x), \ \forall X\in T_xM.
\end{equation*}
\end{definition}
%%%%%%%%%%%%%%%%%%%%%%%%%%%%%%%%%%%%%%%%%%%%%%%%%%%%%%%%%%%%%%%%%%%%%%%%%%%%%%%%%%%%%%%%%%%%%%%%%%%%%%%%%%%%%%%%%%%%%%%%%%%%%%%%%%%%%%%%%%%%%%%%%%%%%%%%%%%%%%%%%%%%%%%%%%%%%%%%%%%%%%%%%%%%%%%%%%%%%%%%%%%%%%%%%%%%%%%%%%%%%%%%%%%%%%%%%%%%%%%%%

\section{\textbf{Killing potential in a manifold with positive Ricci curvature}}
A smooth vector field $X\in\chi(M)$ is a Killing vector field \cite{PE06} if
$$\mathcal{L}_Xg=0, \text{ i.e., }g(\nabla_YX,Z)+g(Y,\nabla_ZX )=0\quad \forall Y,Z\in\chi(M).$$
Killing vector field $X$ has the property that local $1$-parameter group generated by $X$ consists of local isometries. The existence of Killing vector field implies various interesting geometric and topological structures of a Riemannian manifold. For example, every two dimensional Riemannian manifold with a unit Killing vector field is isometric to one of the complete surfaces \cite{BN08}: the Euclidean plane, cylinder, torus, M\"{o}bius strip or Klein bottle. If a complete non-compact Riemannian manifold $M$ with negative semi-definite Ricci curvature admits a non-trivial Killing vector field with finite global norm, then volume of $M$ is finite \cite{YO82}.\\
\indent In 1999, Cr\u{a}\c{s}m\u{a}reanu \cite{CM99} defined the notion of Killing potential in a complete Riemannian manifold. A function $f\in C^\infty(M)$ is called Killing potential if $\nabla f$ is a Killing vector field \cite{CM99}. Cr\u{a}\c{s}m\u{a}reanu also classified the manifolds which admit non-trivial Killing potential \cite[Classification Theorem]{CM99}.\\
In this section we will prove that no non-trivial Killing potential exists in a complete Riemannian manifold with positive definite Ricci curvature.
\begin{theorem}
There does not exist any non-trivial Killing potential $u\in C^2(M)$ in a complete Riemannian manifold $M$ with $Ric_M\geq cg$ for some constant $c>0$.
\end{theorem}
\begin{proof}
Suppose $u\in C^2(M)$ is a non-trivial Killing potential. Now Hessian of $u$ vanishes \cite{CM99}. Since $u$ is a Killing potential so $u$ is harmonic function \cite{CM99}. Hence using Bochner's formula we get
$$\frac{1}{2}\Delta|\nabla u|^2=|Hess(u)|^2+Ric(\nabla u,\nabla u)=Ric(\nabla u,\nabla u).$$
Again $Ric_M\geq cg$ implies that
$$\Delta |\nabla u|^2\geq 2cg(\nabla u,\nabla u)=2c|\nabla u|^2.$$
Since $u$ is Killing potential so $\nabla u$ is covariant constant field \cite{CM99}, i.e., $\nabla_ X(\nabla u)=0 \quad \forall X\in \chi(M)$, which implies that $g(\nabla u,\nabla u)=C,$ for some constant $C> 0$. Hence $\Delta |\nabla u|^2=0$ and it leads to a contradiction. Hence $u$ is constant.
\end{proof} 
\begin{lemma}\cite[Corollary 2]{BN08}
A complete Riemannian manifold $M$ of negative definite Ricci curvature has no nontrivial Killing fields of constant length.
\end{lemma}
Hence we state the following:
\begin{corollary}
If a complete Riemannian manifold $M$ admits a Killing potential $u\in C^\infty(M)$, then $Ric(\nabla u,\nabla u)_p=0\ \forall p\in M$.
\end{corollary}
%%%%%%%%%%%%%%%%%%%%%%%%%%%%%%%%%%%%%%%%%%%%%%%%%%%%%%%%%%%%%%%%%%%%%%%%%%%%%%%%%%%%%%%%%%%%%%%%%%%%%%%%%%%%%%%%%%%%%%%%%%%%%%%%%%%%%%%%%%%%%%%%%%%%%%%%%%%%%%%%%%%%%%%%%%%%%%%%%%%%%%%%%%%%%%%%%%%%%%%%%%%%%%%%%%%%%%%%%%%%%%%%%%%%%%%%%%%%%%%

\section{\textbf{Some inequalities of harmonic functions in a complete Riemannian manifold} }

\indent In this section, we have proved that if a harmonic function attains negative value at a point in a complete Riemannian manifold then the integral of that function ramains invariant in any ball with center at that particular point. The main theorem of this section is as follows:
\begin{theorem}[Integral Liouville Theorem]\label{th2}
Let $M$ be a complete Riemannian manifold and $u:M\rightarrow\mathbb{R}$ be a harmonic function. If there exists a point $p\in M$, where $M$ is rotationally symmetric and $u(p)<0$ with $inj(p)>1$, then 
\begin{equation*}
\int_{B_r}(K+u)dV=0\quad \forall r\in [0,inj(p)),
\end{equation*}
for some constant $K>0$.
\end{theorem}
\begin{proof}
Let $p\in M$ be a fixed point and $B_r$ be an open geodesic ball with center at $p$ and radius $r$ such that $0\leq r\leq inj(p)$ and also $inj(p)>1$. Now for each $r>0$, we define
\begin{equation}\label{eq1}
v(r)=\frac{1}{Vol(\partial B_r)}\int_{\partial B_r}udS_r,
\end{equation}
where $dS_r$ is the $(n-1)$-dimensional volume element of $B_r$. Since $M$ is rotationally symmetric at $p$ so using (\ref{eq12}) and (\ref{eq1}), we get
\begin{equation*}
v(r)=\frac{1}{Vol(\partial B_1)}\int_{\partial B_1}u(r\eta)dS_1.
\end{equation*}
Now differentiating the above function with respect to $r$, we obtain
$$v'(r)=\frac{1}{Vol(\partial B_1)}\int_{\partial B_1}\frac{\partial u}{\partial r}(r\eta)dS_1=\frac{1}{Vol(\partial B_r)}\int_{\partial B_1}\partial_rudS_r.$$
Hence for $r>0$, by using divergence theorem, we get
\begin{equation*}
v'(r)Vol(\partial B_r)=\int_{B_r}\Delta u dV.
\end{equation*}
In view of harmonicity of $u$, the above equation implies that $v'(r)=0$, i.e., $v=C$, a constant. Since $v$ is continuous so $v$ is constant in $[0,inj(p))$. Hence (\ref{eq1}) implies that 
\begin{equation}\label{eq2}
\int_{\partial B_r}udS_r=CVol(\partial B_r).
\end{equation}
Since $u(p)<0$, then there exists $\epsilon>0$, such that $u<0$ in $\partial B_\epsilon$. Hence
\begin{equation*}
C=\frac{1}{Vol(\partial B_\epsilon)}\int_{\partial B_\epsilon}udS_\epsilon<0.
\end{equation*}
Suppose $C=-K$, where $K>0$. Hence 
$$Vol(\partial B_r)=-\frac{1}{K}\int_{\partial B_r}udS_r.$$
Now from (\ref{eq2}), we get
\begin{equation}\label{eq5}
Vol(B_r)=\int_{0}^{r}Vol(\partial B_t)dt=-\frac{1}{K}\int_{0}^{r}\int_{\partial B_t}udS_tdt=-\frac{1}{K}\int_{B_r}udV.
\end{equation}
The above equation implies that
$$\int_{B_r}dV+\frac{1}{K}\int_{B_r}udV=0,$$
i.e.,
\begin{equation*}
\int_{B_r}(K+u)dV=0,\quad \forall r\in [0,inj(p)).
\end{equation*}
\end{proof}

\begin{corollary}
Suppose $(M,o)$ is a Riemannian manifold with a pole $o$ and is rotationally symmetric at $o$. If $u:M\rightarrow\mathbb{R}$ is a harmonic function such that $f(o)<0$, then there exists a constant $K>0$ such that
\begin{equation}\label{eq3}
\int_{B_r(o)}(K+u)dV=0,\quad \forall r\in [0,\infty).
\end{equation}
\end{corollary}
\begin{theorem}
Assuming the conditions of the above corollary with $Ric_M\geq 0$, there exists a constant $C$ such that
\begin{equation}
\int_{B_r(o)}(C+u)dV=0,\quad \forall r\in [0,\infty).
\end{equation}
\end{theorem}
\begin{proof}
Let $u\in C^2(M)$ be a harmonic function. Since a manifold with positive semi-definite Ricci curvature does not admit any non-constant positive harmonic function \cite[p. 21]{SY94}, so there exists a point $q\in M$ such that $u(q)\leq 0$. Without loss of generality we can take $u(q)<0$, otherwise we have to subtract some positive constant. If $u(o)<0$, then taking $p=o$ the proof is complete. If $u(o)\geq 0$, then take $f(x)=u(x)-u(o)$ for $x\in M$. Hence $f$ is harmonic and $f(o)<0$. So using above corollary we get there is a constant $K$ such that for all $r>0$ 
\begin{eqnarray*}
0&=& \int_{B_r(o)}(K+f)dV,\\
&=&\int_{B_r(o)}(K+u(x)-u(o))dV.
\end{eqnarray*}
Hence taking $C=K+u(o)$, we get our result.
\end{proof}
\begin{corollary}
Let $(M,o)$ be a Riemannian manifold with $Ric_M\geq 0$. Suppose $M$ is rotationally symmetric at $o$. Then all non-constant convex functions $u\in C^2(M)$ with $|\nabla u|\in L^1(M)$ satisfy \begin{equation*}
\int_{B_r}(K+u)dV=0,\quad \forall r\in [0,\infty),
\end{equation*}
for some constant $K$.
\end{corollary}
\begin{proof}
Since $u$ is convex so $u$ is subharmonic. Since $v$ has integrable norm so $u$ is harmonic \cite{YAU76}. Hence the proof follows from the above Theorem.
\end{proof}
\begin{definition}\cite{CM97}
A complete Riemannian manifold $M$ is said to satisfy a uniform Neumann-Poincar\'{e} inequality if there exists $C_N<\infty$ such that, for all $p\in M$, $r>0$ and $u\in W^{1,2}_{loc}(M),$
\begin{equation*}
\int_{B_r}\Big(u-\frac{1}{Vol(B_r)}\int_{B_r}udV \Big)^2dV\leq C_Nr^2\int_{B_r}|\nabla u|^2dV.
\end{equation*}
\end{definition}
Yau \cite{YAU75} developed the method of gradient estimate of harmonic function in a manifold and proved that for any complete Riemannian manifold with $Ric_M>-K$, where $K\geq 0$ is a constant,  every harmonic function $u$ which is bounded below satisfies the following inequality
\begin{equation}\label{eq8}
|\nabla u|\leq \sqrt{(n-1)K}(u-\inf_Mu).
\end{equation}
Now integrating the above inequality in $B_r$ and using (\ref{eq5}) and $\inf_Mu=-c<0$, for some $c>0$, we obtain
\begin{equation}\label{eq10}
\int_{B_r}|\nabla u|^2dV\leq (n-1)K\int_{B_r}(u+c)dV.
\end{equation}
Now suppose that $M$ satisfies a uniform
Neumann-Poincar\'{e} inequality and also $u\in W^{1,2}_{loc}(M)$, then
\begin{equation*}
\int_{B_r}\Big(u-\frac{1}{Vol(B_r)}\int_{B_r}udV \Big)^2dV\leq C_Nr^2\int_{B_r}|\nabla u|^2dV,
\end{equation*}
for some constant $C_N<\infty$.
Now (\ref{eq10}) implies that
\begin{equation*}
\int_{B_r}\Big(u-\frac{1}{Vol(B_r)}\int_{B_r}udV \Big)^2dV\leq (n-1)KC_Nr^2\int_{B_r}(u+c)^2dV.
\end{equation*}
Again using Theorem \ref{th2}, the above inequality implies that, there is a constant $C>0$ such that
\begin{equation}
\int_{B_r}\Big(u+C \Big)^2dV\leq (n-1)KC_Nr^2\int_{B_r}(u+c)^2dV.
\end{equation}
Hence we arrive at the following conclusion:
\begin{proposition}
Let $M$ be a $n$-dimensional complete Riemannian manifold with $Ric_M>-K$, for $K>0$ and $u\in W^{1,2}_{loc}(M)$ be a harmonic function and also $M$ satisfies a uniform
Neumann-Poincar\'{e} inequality. If $u$ is bounded below and there is a point $p\in M$ where $M$ is rotationally symmetric and such that $u(p)<0$ with $inj(p)>1$, then there exist $C>0,C_N$ and $c>0$, depending only on $n$, such that, for $0\leq r<inj(p)$, the following inequality holds 
\begin{equation}
\int_{B_r}\Big(u+C \Big)^2dV\leq (n-1)KC_Nr^2\int_{B_r}(u+c)^2dV.
\end{equation}
\end{proposition}
%%%%%%%%%%%%%%%%%%%%%%%%%%%%%%%%%%%%%%%%%%%%%%%%%%%%%%%%%%%%%%%%%%%%%%%%%%%%%%%%%%%%%%%%%%%%%%%%%%%%%%%%%%%%%%%%%%%%%%%%%%%%%%%%%%%%%%%%%%%%%%%%%%%%%%%%%%%%%%%%%%%%%%%%%%%%%%%%%%%%%%%%%%%%%%%%%%%%%%%%%%%%%%%%%%%%%%%%%%%%%%%%%%%%%%%%%%%%%%%%%
\section{\textbf{Harmonic function and gradient Ricci soliton}}
A Riemannian manifold $(M,g)$ is said to be a Ricci soliton generator if there exists a smooth vector field $X\in \chi(M)$ satisfying 
\begin{equation}
\mathcal{L}_{X}g+Ric=cg,
\end{equation}
where $c$ is a constant. Ricci soliton generators are natural generalization of Einstein metrics and it is a special solution of Ricci flow \cite{CA06}. A Ricci soliton generator is shrinking soliton, steady soliton or expanding soliton according as $c>0$, $c=0$ or $c<0$.
Let $(M,g)$ be a complete manifold. Then it is said to be complete gradient Ricci soliton if there exists a smooth function $u:M\rightarrow\mathbb{R}$ such that 
\begin{equation}
\mathcal{L}_{\nabla u}g+Ric=cg.
\end{equation}
Gradient Ricci solitons have been studied quite extensively in last decade, see(\cite{PW10}, \cite{MM08}).\\
\indent In this section we have studied general form of gradient Ricci soliton generator where the potential function is harmonic and proved that the harmonic function takes a simple form along some geodesic.
\begin{theorem}\label{th1}
Let  $(M,g)$ be a complete non-compact Riemannian manifold with $Ric_M>-K$, where $K>0$. If there exists a harmonic function $u\in C^2(M)$ which is bounded below and satisfies
\begin{equation}\label{eq6}
\mathcal{L}_{\nabla u}g+Ric\geq cg,
\end{equation}
then for some geodesic $\gamma:[0,\infty)\rightarrow M$, the function $u$ takes the form
\begin{equation}\label{eq9}
u(\gamma(t))=\frac{1}{\sqrt{(n-1)K}}(ct+G), \quad \forall 0\leq t<\infty,
\end{equation}
for some constant $G$.
\end{theorem}
\begin{proof}
Since $M$ is not compact, so by Ambrose's compactness criteria \cite{AM57} there exists a geodesic $\gamma:[0,\infty)\rightarrow M$, parametrized by arc length, such that 
\begin{equation}\label{eq7}
\int_{0}^{\infty}Ric(\gamma',\gamma')dt<\infty.
\end{equation}
Then taking $X=\nabla u$, along $\gamma$ we have
$$\mathcal{L}_Xg(\gamma',\gamma')=2g(\nabla_{\gamma'}X,\gamma')=2\frac{d}{dt}[g(X,\gamma')].$$
Hence using above equation and (\ref{eq6}), we obtain
$$Ric(\gamma',\gamma')\geq cg(\gamma',\gamma')-2\frac{d}{dt}[g(X,\gamma')].$$
Now taking $\gamma(0)=q$ and integrating and using the Cauchy-Schwarz inequality, we get
\begin{eqnarray*}
\int_{0}^{T}Ric(\gamma',\gamma')dt &\geq & c\int_{0}^{T}g(\gamma',\gamma')dt-2\int_{0}^{T}\frac{d}{dt}[g(X,\gamma')]dt\\
&\geq & cT+2g(X_q,\gamma'(0))-2g(X_{\gamma(T)},\gamma'(T))\\
&\geq & cT+2g(X_q,\gamma'(0))-2\norm{X_{\gamma(T)}}.
\end{eqnarray*}
Now from (\ref{eq7}), we get
\begin{equation*}
\lim\limits_{T\rightarrow\infty}(cT+2g(X_q,\gamma'(0))-2\norm{X_{\gamma(T)}})<\infty.
\end{equation*}
Hence 
\begin{equation}
\lim\limits_{T\rightarrow\infty}(cT-2\norm{X_{\gamma(T)}})<\infty.
\end{equation}
Since $Ric_M>-K$ and $u$ is harmonic with lower bound, then gradient estimate Theorem (\ref{eq8}) implies that
\begin{equation*}
\lim\limits_{T\rightarrow\infty}(cT-2\sqrt{(n-1)K}(u-\inf_Mu))<\infty.
\end{equation*}
Consequently
\begin{equation*}
\lim\limits_{T\rightarrow\infty}(cT-2u\sqrt{(n-1)K})<\infty.
\end{equation*}
Hence there exists a constant $G$ such that
\begin{equation}
u(\gamma(t))=\frac{1}{\sqrt{(n-1)K}}(ct+G), \quad \forall 0\leq t<\infty,
\end{equation}
\end{proof}
There are some immediate consequences of the above theorem.
\begin{corollary}
Suppose $(M,g)$ is a complete non-compact Riemannian manifold with $Ric_M>-K$. If there is non-positive harmonic function $u\in C^2(M)$ satisfying (\ref{eq6}) for some positive constant $c$, then $u$ is constant. Hence $M$ satisfies $Ric\geq cg.$
\end{corollary}
\begin{corollary}
Let $M$ be a complete Riemannian manifold with $Ric_M>-K$. If there is a convex function $u\in C^2(M)$ which is bounded below and $|\nabla u|\in L^1(M)$ and also satisfies (\ref{eq6}), then there is a geodesic $\gamma:[0,\infty)\rightarrow M$, where $u$ takes the form (\ref{eq9}).
\end{corollary}
\begin{proof}
The convexity of $u$ implies that $u$ is subharmonic \cite{GW71}. Yau \cite{YAU76} proved that any subharmonic function $u\in C^2(M)$ with integrable norm on a complete Riemannian manifold is harmonic. Hence this corollary follows from Theorem \ref{th1}.
\end{proof}

\section*{\textbf{Acknowledgment}}
 The second author greatly acknowledges to
The University Grants Commission, Government of India for the award of Junior Research
Fellowship.

$\bigskip $

$^{1,2}$The University of Burdwan, 

Department of Mathematics,

 Golapbag, Burdwan-713104,
 
 West Bengal, India.
 
$^1$E-mail:aask2003@yahoo.co.in, aashaikh@math.buruniv.ac.in

$^2$E-mail:chan.alge@gmail.com

\end{document}